\documentclass[11pt,a4paper]{article}
\usepackage{amsthm,amsmath,amsfonts,amssymb}
\usepackage[english]{babel}
\usepackage{graphicx}
\usepackage{verbatim}
\usepackage{pgf}
\usepackage[T1]{fontenc}
\usepackage[utf8x]{inputenc}
\usepackage{hyperref}

\theoremstyle{plain}   
\newtheorem{theorem}{Theorem}
\newtheorem{definition}{Definition}
\newtheorem{proposition}{Proposition}
\newtheorem{corollary}{Corollary}

\newtheorem{question}{Question}

\newtheorem{problem}{Problem}
\newtheorem{observation}{Observation}

\newcommand{\dist}{\mathrm{dist}}
\newcommand{\dg}{\mathrm{deg}}

\begin{document}

\title{On new record graphs close to bipartite Moore graphs}

\author{Gabriela Araujo \\
\small Instituto de Matem\'aticas \\[-0.8ex]
\small Universidad Nacional Autonóma de México \\[-0.8ex]
\small M\'exico D.F. 04510,M\'exico\\[-0.8ex]
\small \texttt{garaujo@matem.unam.mx}\\ [2ex]
Nacho L\'opez \\
\small Departament de Matem\`atica\\[-0.8ex]
\small Universitat de Lleida\\[-0.8ex]
\small Jaume II 69, 25001 Lleida, Spain\\[-0.8ex]
\small \texttt{nacho.lopez@udl.cat}\\
}
\date{\today}
\maketitle

\begin{abstract}
The modelling of interconnection networks by graphs motivated the study of several extremal problems that involve well known parameters of a graph (degree, diameter, girth and order) and ask for the optimal value of one of them while holding the other two fixed. Here we focus in {\em bipartite Moore graphs\/}, that is, bipartite graphs attaining the optimum order, fixed either the degree/diameter or degree/girth. The fact that there are very few bipartite Moore graphs suggests the relaxation of some of the constraints implied by the bipartite Moore bound. First we deal with {\em local bipartite Moore graphs}. We find in some cases those local bipartite Moore graphs with local girths as close as possible to the local girths given by a bipartite Moore graph. Second, we construct a family of $(q+2)$-bipartite graphs of order $2(q^2+q+5)$ and diameter $3$, for $q$ a power of prime. These graphs attain the record value for $q=9$ and improve the values for $q=11$ and $q=13$.\end{abstract}



\section{Introduction}

The {\emph{degree/diameter problem}} for graphs consists in finding the largest order of a graph with prescribed degree and diameter (for a survey of it see \cite{MS16}). Here we focus on the class of bipartite graphs. In this context, given the values of maximum degree $\Delta$ and diameter $d$ of a bipartite graph, there is a natural upper bound for its number of vertices $n$,
\begin{equation}\label{eq:bipmoore_bound}
n\leq M^b_{\Delta,d} =
\left\{
\begin{array}{ccc}
2\dfrac{(\Delta-1)^d-1}{\Delta-2} & \textrm{ if } & \Delta > 2, \\
2d & \textrm{ if } & \Delta=2.
\end{array}
\right.
\end{equation}

where $M^b_{\Delta,d}$ is known as the {\em bipartite Moore bound\/}. It is well known that biparite Moore graphs only exists for a few combinations of the parameters $\Delta$ and $d$, namely, $\Delta=1$ and $d=1$ (Complete graph of two vertices), $\Delta = 2$ and $d \geq 1$ (Cycle graphs),  $d = 2$ and $\Delta \geq 3$ (Complete bipartite graphs). For $d = 3, 4, 6$ bipartite Moore graphs have been constructed only when $\Delta-1$ is a prime power (see \cite{MS16}). \\

The fact that there are very few Moore graphs suggested the study of graphs `close' to the Moore ones. This `closeness' has been usually measured as the difference between the (unattainable) Moore bound and the order of the considered graphs. In this sense, the existence and construction of graphs with small `defect' $\delta$ (order $n=M(\Delta,d)-\delta$) has deserved much attention in the literature (also see \cite{MS16} and \cite{UPCpage} for complete surveys of the problem). \\

Besides, the {\em degree\/}/{\em girth problem\/} (also known as {\em the cage problem\/}) consists in finding the smallest order of a graph with prescribed degree and girth (for a survey of it see \cite{ExooJaj08}). In this context, given the values of the maximum degree $\Delta>2$ and the girth $g$ of a bipartite graph (notice that $g$ must be even) there is a natural lower bound for its number of vertices $n$,
\begin{equation}\label{eq:bipgirthmoore_bound}
n\geq M^b_{\Delta,g} = 2\big(1+(\Delta-1)+(\Delta-1)^2+\cdots+(\Delta-1)^{\frac{g-2}{2}}\big)=2\dfrac{(\Delta-1)^{g/2}-1}{\Delta-2}
\end{equation}

graphs attaining such a bound are again biparite Moore graphs and hence both problems collide in their corresponding extremal value ($g=2d$). Both bounds \eqref{eq:bipmoore_bound} and \eqref{eq:bipgirthmoore_bound} can easily calculated just by counting the number of vertices at every distance from any edge $uv$ in a {\em bipartite Moore tree} (see Figure \ref{fig:moore_graph}) where the depth of the tree is $d$ (and therefore the graph must have girth $g=2d$) and where every vertex has degree $\Delta$ (except those vertices at depth $d$, that is, terminal vertices of the tree).

\begin{figure}[htb]
        \centerline{\includegraphics*[scale=0.6]{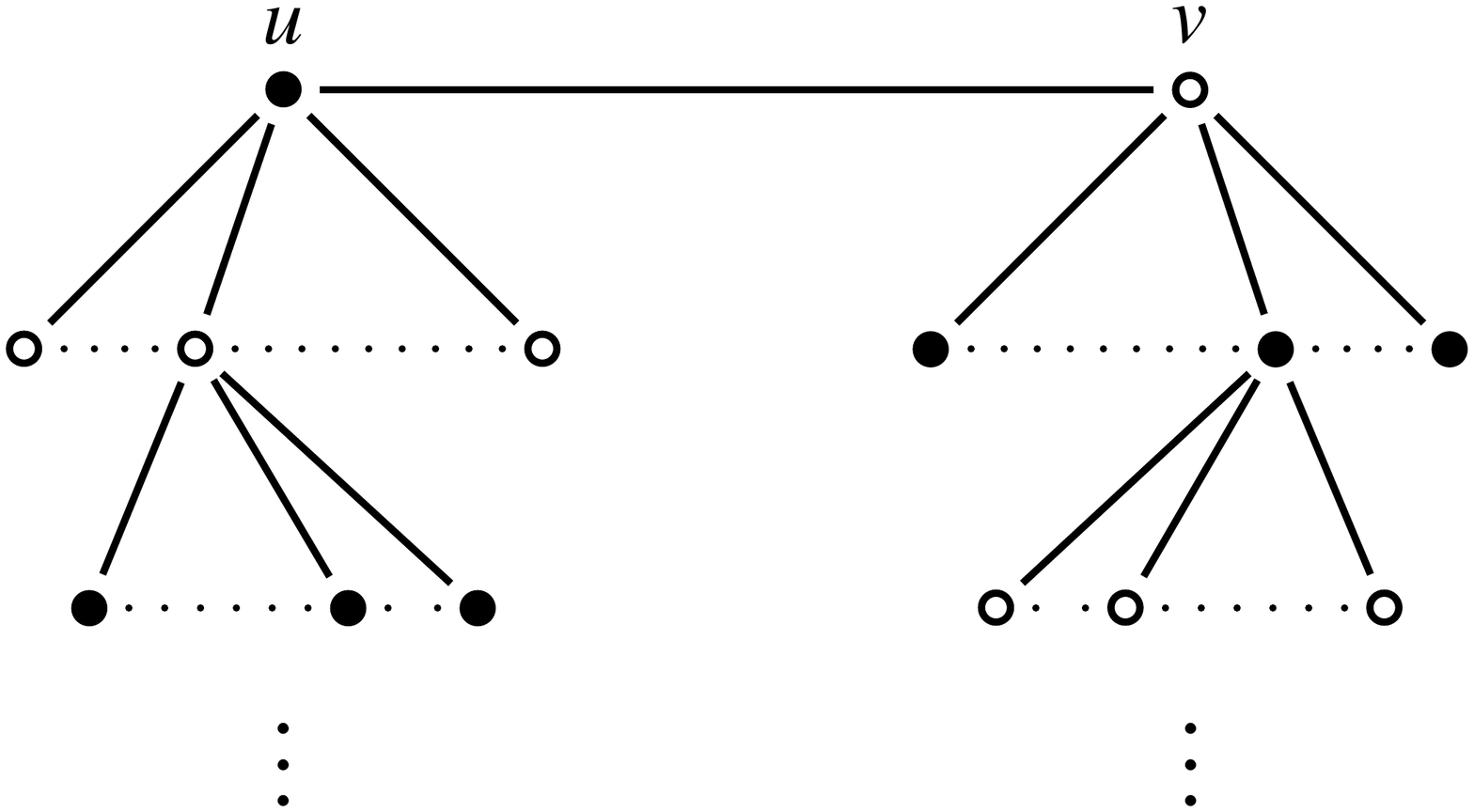}}
        \caption{Bipartite Moore tree, that is, a distance-preserving spanning tree of a bipartite Moore graph depicted `hanging' from any of its edges $uv$.}
        \protect\label{fig:moore_graph}
\end{figure}

Also the existence of graphs with small `excess' $\epsilon$ (order $n=M^b_{\Delta,g} +\epsilon$) has been studied deeply and those graphs with prescribed degree/girth and minimum order are known as $(\Delta,g)$-{\em cages}. \\

Another kind of approach considers relaxing some of the constraints implied by the Moore bound. For instance, this approach has been considered yet in the context of the degree/diameter problem, where Tang, Miller and Lin \cite{TML08} relax the condition of the degree and admit few vertices with degree $\Delta+\epsilon$. Alternatively, Capdevila et al. \cite{Cap10} allow the existence of vertices with eccentricity $d+1$. In this context, regular graphs of degree $d$, radius $d$, diameter $d+1$ and order equal to the Moore bound are known as {\em radial Moore graphs\/}.

\subsection*{Contributions in this paper}



We study two problems related with the `relaxation' of bipartite Moore graphs. In section \ref{sec:measure}, inspired in the idea given by Capdevila et al \cite{Cap10} about radial Moore graphs (relaxing the diameter of the graph), we define the local bipartite Moore graphs as graphs preserving one structural property that every bipartite Moore graph has (in this case we relax the girth). We enumerate these extremal graphs in a special case and we rank them according to their proximity to a bipartite Moore graph.

We attend a problem related with the classical relaxation of `closeness' in section \ref{sec:gabi}, constructing  bipartite graphs with small defect. In particular, we construct a $(q+2)$-regular bipartite graph of diameter $3$ and girth $4$ adding $8$ vertices to the incidence graph of the projective plane of order $q$. This result presents new record graphs in the context of the degree/diameter problem, to consult the record regular bipartite graphs known until this moment see \cite{table}. Moreover, the graphs attained and improved with our construction are given in (\cite{BDF81,DF84,FPMPV13}).

Finally, in section \ref{sec:openproblems}, we present several questions and open problems related to this work.

\subsection*{Terminology and notation}

Let $G=(V,E)$ be a connected graph. Given two vertices $u$ and $v$ of $G$, the distance between $u$ and $v$, $\dist_G(u,v)$, is the length of a shortest path joining them. The  {\em girth\/} of a graph $G$ is the length of its shortest cycle. If we restrict our attention to the cycles through a given vertex $v$, we can define the {\em local girth\/} of vertex $v$, $g(v)$, as the smallest length of such `rooted' cycles. The vector $\mathbf{g}(G)$ constituted by the girths of all its vertices will be referred to as the {\em girth vector\/} of $G$. Usually, when the vector is long enough, we denote it with a short description using superscripts, that is, $\mathbf{g}(G): g_1^{n_1},g_2^{n_2},\dots,g_k^{n_k}$, where $g_1>g_2> \dots >g_k$, and $n_i$ denotes the number of vertices having $g_i$ as its local girth, for all $1 \leq i \leq k$.  
\section{Local bipartite Moore graphs}\label{sec:measure}

Radial Moore graphs (see \cite{Cap10,EGLG12,GM15}) are defined as regular graphs of degree $\Delta$, order $M^b_{\Delta,d}$, radius $d$ and diameter $d+1$. In other words, radial Moore graphs have the same distance-preserving spanning tree than a Moore graph has for any vertex, but only for some of their vertices (those vertices with eccentricity $d$). In a radial Moore graph it is allowed the existence of other spanning trees of depth at most $d+1$. Here we do something similar with bipartite Moore graphs: if we hang a bipartite Moore graph from any of its edges $uv$ we observe the `same' distance-preserving spanning tree (see figure \ref{fig:moore_graph}). Thus we could relax this property in the set of regular bipartite graphs by forcing the existence of this Moore tree for at least one edge, but allowing others distance-preserving spanning trees in the graph for other edges. This is how we define a {\em local bipartite Moore graph}.

\begin{definition}
Given two positive integers $\Delta \geq 2$ and $g \geq 4$, a connected regular bipartite graph of degree $\Delta$ and order $M^b_{\Delta,g}$ is said to be a local bipartite Moore graph if it contains at least one edge such that its corresponding distance-preserving spanning tree is a bipartite Moore tree (see figure \ref{fig:moore_graph}).  
\end{definition}

From this point of view, bipartite Moore graphs are a particular case of local bipartite Moore graphs having the same distance-preserving spanning tree for any of its edges. There are at least two vertices with local girth $g$ in a local bipartite Moore graph (the end vertices of the edge $uv$ in figure \ref{fig:moore_graph}), meanwhile for a Moore graph all vertices have local girth $g$ (and hence the whole graph has girth $g$). \\

Let us denote by ${\cal RB}(\Delta,g)$ the set of all nonisomorphic regular bipartite graphs of degree $d$ and order $M^b_{\Delta,g}$. The set of local bipartite Moore graph of degree $\Delta$ and girth $g$ will be denoted as ${\cal LBM}(\Delta,g)$. Of course ${\cal LBM}(\Delta,g) \subseteq {\cal RB}(\Delta,g)$ and our purpose is give some extra information about this set. \\

Let us start with the case $g=4$ and any $\Delta \geq 2$. Let $G \in {\cal LBM}(\Delta,4)$, then, according to Eq. \eqref{eq:bipgirthmoore_bound}, $G$ has order $M^b_{\Delta,4}=2\Delta$ and $G$ contains the Moore tree depicted in Fig. \ref{fig:bmoore_graphg4}. Due to the regularity of $G$, every vertex of the set $\{u_2,\dots,u_{\Delta}\}$ is adjacent to each vertex of $\{v_2,\dots,v_{\Delta}\}$ and viceversa. Hence local bipartite Moore graphs for $g=4$ are complete bipartite graphs $K_{\Delta,\Delta}$ and they are indeed Moore graphs (both for girth $g=4$ or diameter $d=2$).
\begin{figure}[htb]
\begin{center}
\begin{tabular}{c}
\includegraphics*[scale=0.6]{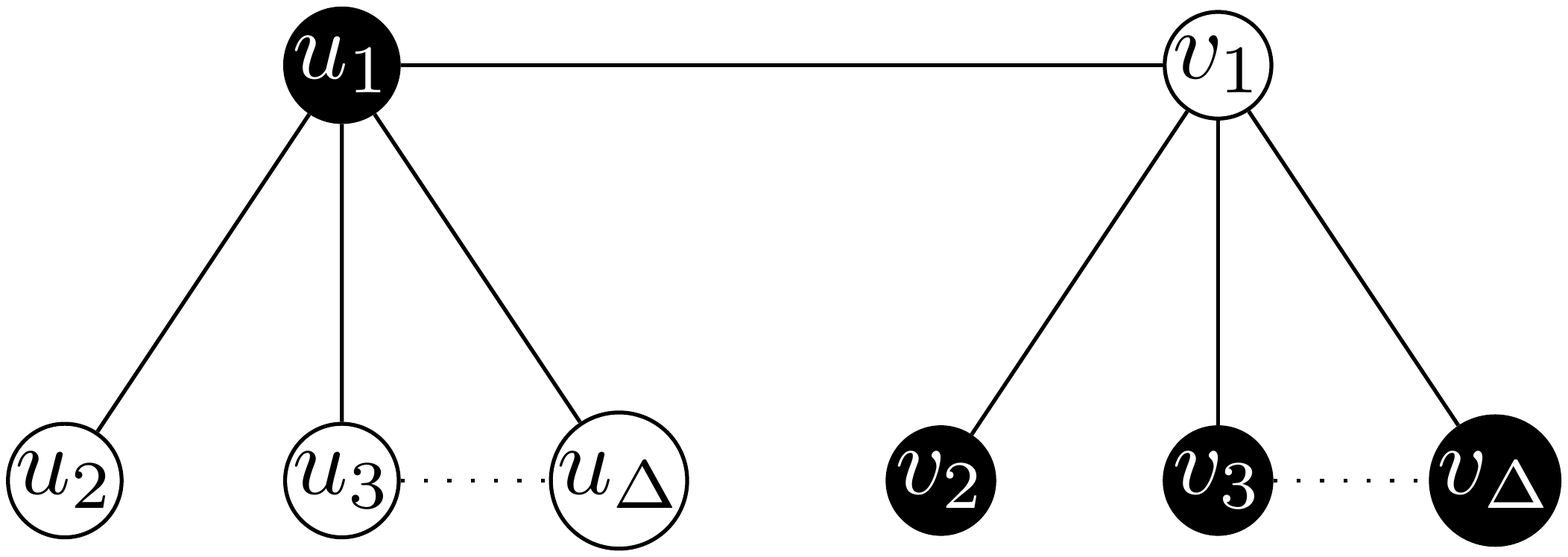} 
\end{tabular}
\end{center}
        \caption{Spanning tree of a bipartite Moore graph for $g=4$.}
        \protect\label{fig:bmoore_graphg4}
\end{figure}

The case $g=6$ requires a deeper analysis. Next proposition gives a complete enumeration of these graphs for the cubic case $\Delta=3$.

\begin{proposition}\label{prop:enum}
There are $5$ graphs in the set ${\cal LBM}(3,6)$, namely $H,G_1,G_2,G_3$ and $G_4$ depicted in figures \ref{fig:examp1} and \ref{fig:examp2}, where $H$ is the unique  bipartite Moore graph for these parameters (also known as the Heawood graph). 
\end{proposition}

\begin{proof}
 Let $G$ be a graph $ G \in {\cal LBM}(3,6)$. By definition, $G$ contains the distance-preserving spanning tree depicted in Figure \ref{fig:examp1} and we label the vertices as they are depicted in this figure ($B$ for `black' vertices and $W$ for `white' vertices). Notice that the edges of $G$ that do not appear in this Moore tree must join vertices in the set $S=\{B_1,B_2,B'_1,B'_2,W_1,W_2,W'_1,W'_2\}$, that is, those vertices at maximum distance from the `root' edge $(W_r,B_r)$. Moreover, the subgraph $G_S$ of $G$ induced by the set of vertices $S$ must be $2$-regular, and since it is bipartite, then $G_S$ is either a cycle graph of length $8$ or the union of two cycle graphs of order $4$. Hence, in order to complete graph $G$, we have to analyze several cases:
 \begin{figure}[htb]
\begin{tabular}{ccc}
\includegraphics*[scale=0.4]{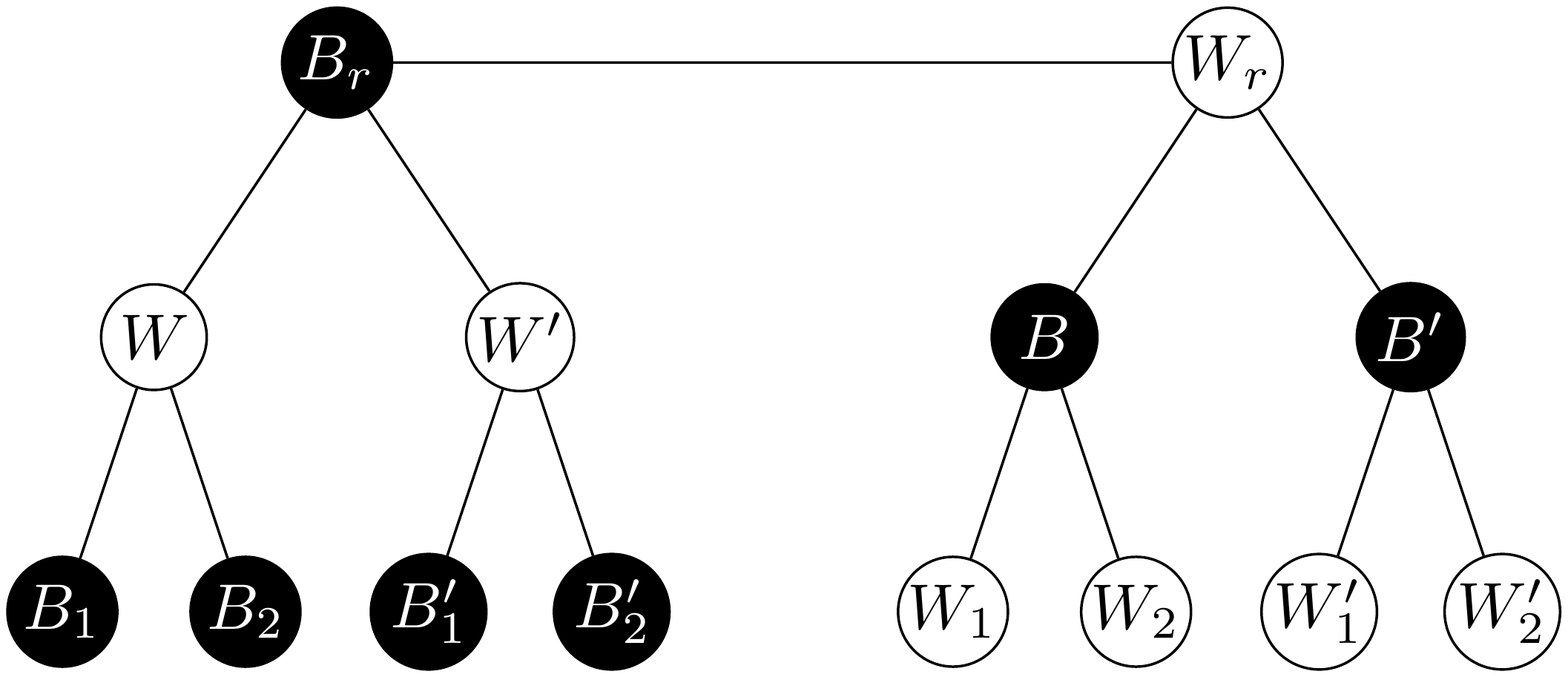} & & \includegraphics*[scale=.3]{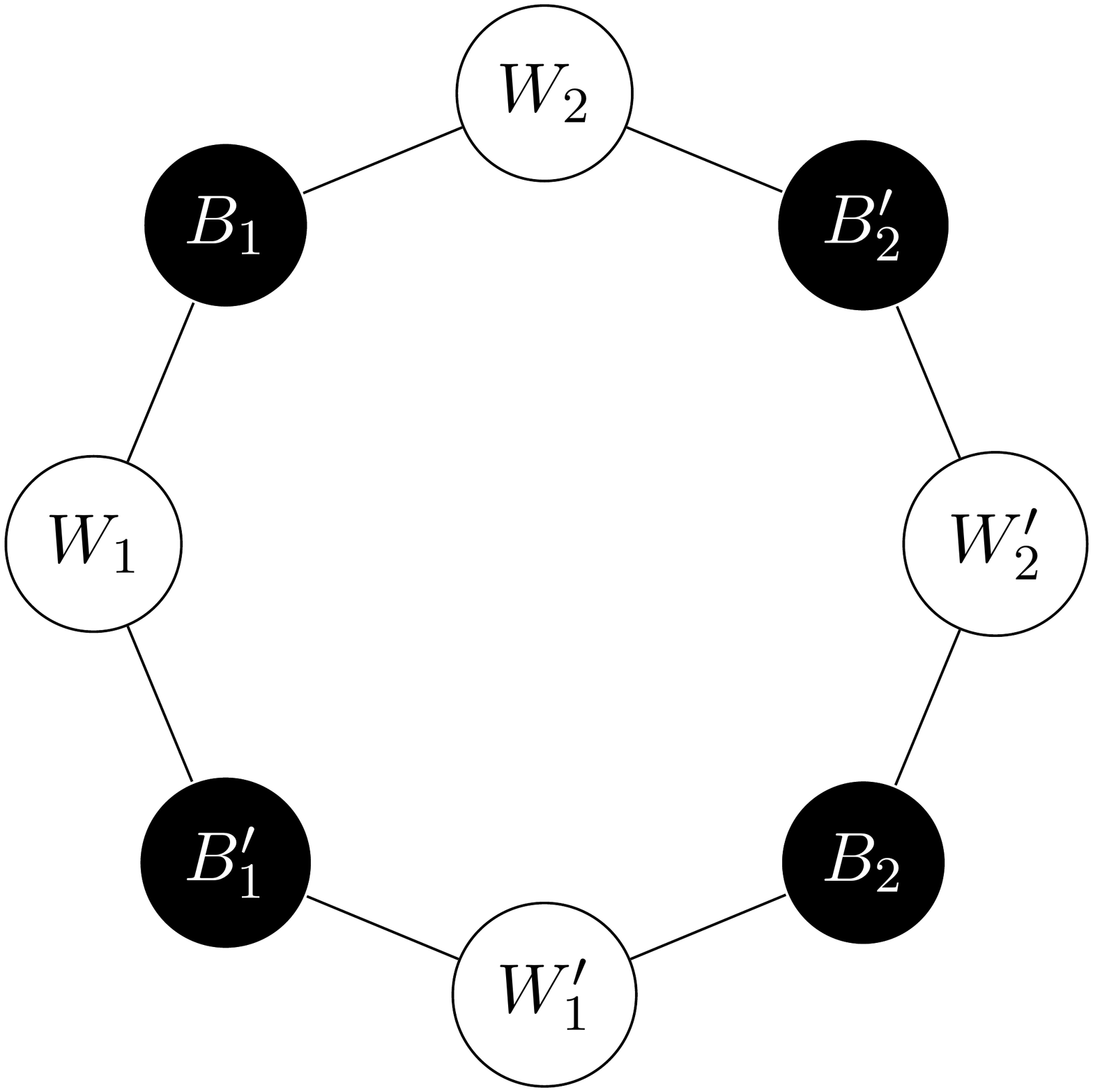} \\
(a) & & (b)
\end{tabular}
        \caption{(a) Spanning tree of a cubic bipartite Moore graph for $g=6$ where vertices are labelled as in Proposition \ref{prop:enum}. (b) The subgraph $G_S$ induced by the vertices at maximum distance from the edge root $(B_r,W_r)$ for the case (a1) in Proposition \ref{prop:enum}.}
        \protect\label{fig:bipcubmoore_graph}
\end{figure}
 \begin{itemize}
  \item[(a)] $G_S=C_8$: We divide this case according to the distance between vertices inside the cycle:
  \end{itemize}
  
  \begin{itemize}
   \item[(a1)] If $\dist_{G_S}(B_1,B_2)=4$ and $\dist_{G_S}(W_1,W_2)=2$. \\
   Then, $\dist_{G_S}(B'_1,B'_2)=4$ and $\dist_{G_S}(W'_1,W'_2)=2$. There are several labelings of $C_8$ with these conditions, one of them is depicted in Fig. \ref{fig:bipcubmoore_graph}. Nevertheless, all of them are equivalent performing a convinient relabeling of vertices in $G$. However, the local girth for all $v \in \{W_1,W_2,W'_1,W'_2,B,B'\}$ is $g(v)=4$, since they belong to a cycle of length $4$ in $G$. Moreover, since $\dist_{G_S}(W_1,W_2)=2$ then $W_1$ and $W_2$ must be adjacent both to just one black vertex in $G_S$ (vertex $B_1$ in Fig. \ref{fig:bipcubmoore_graph}). Hence this black vertex has also local girth $4$. Another one black vertex has also local girth $4$ because $\dist_{G_S}(W'_1,W'_2)=2$. It is easy to see that the remaining vertices have maximum local girth. Hence ${\bf g}(G): 6^6,4^8$. A complete representation of graph $G$ is given in Fig. \ref{fig:prop1}, where subgraph $G_S$ is depicted in red color.

   \begin{figure}[htb]
\centering
\begin{tabular}{ccc}
\includegraphics[width=.45\textwidth]{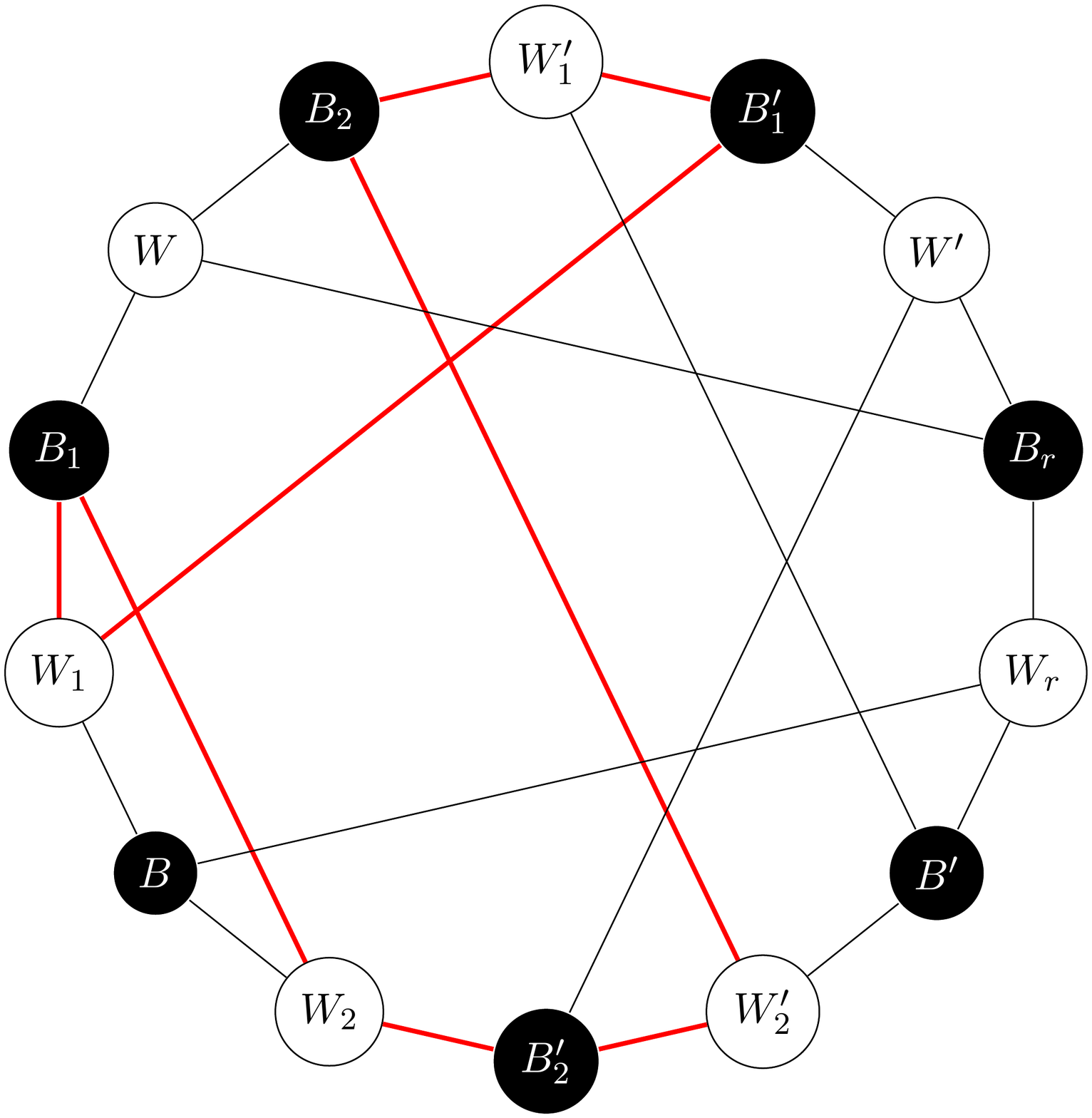} & & \includegraphics[width=.45\textwidth]{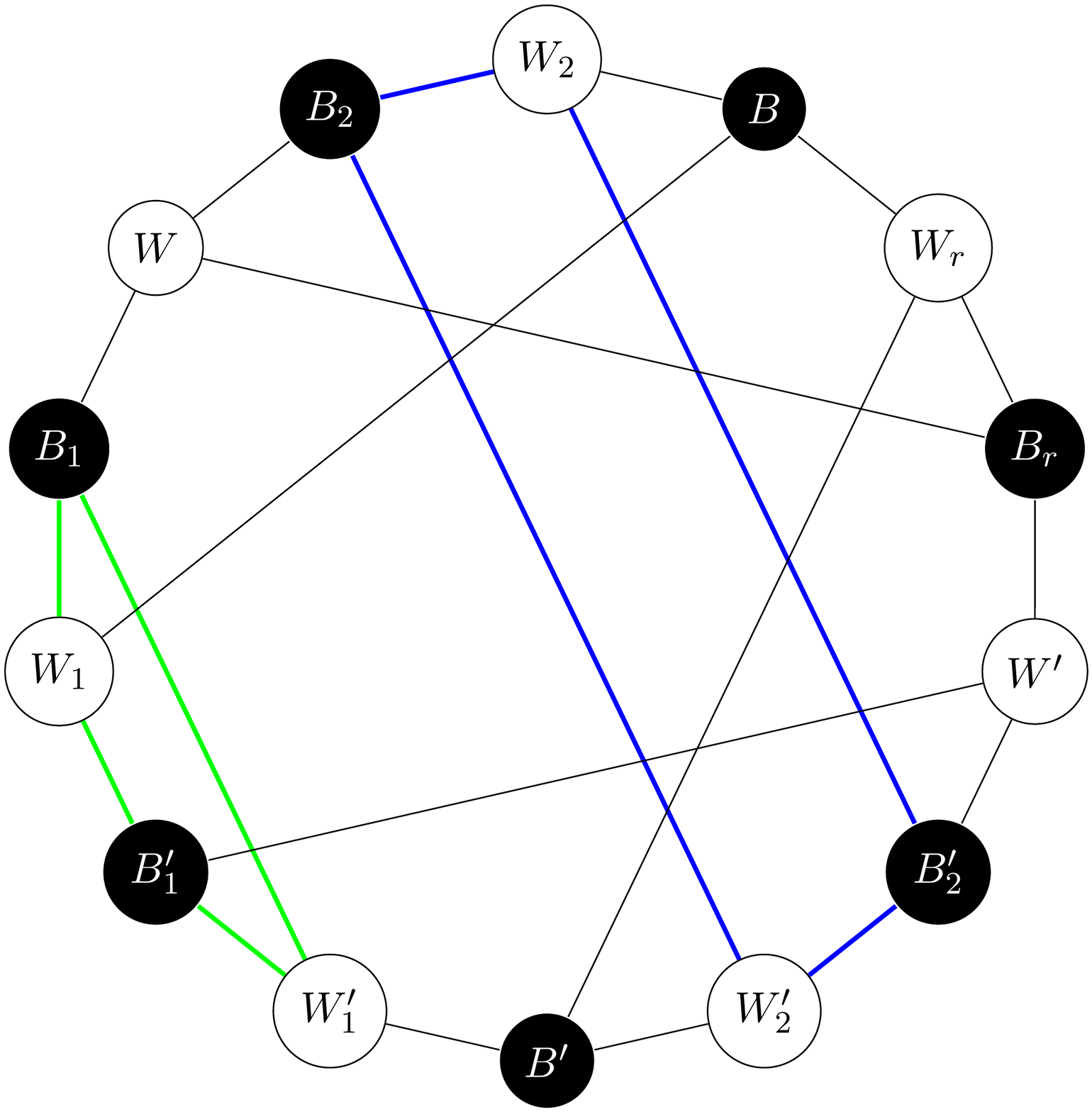} \\
(a1) & & (b1) \\
\end{tabular}
\caption{(a) The graphs obtained in cases (a1) and (b1) of Proposition \ref{prop:enum}. They are isomorphic to the graph $G_1$ depicted in figure \ref{fig:examp1}}.\label{fig:prop1}
\end{figure}

   \item[(a2)] If $\dist_{G_S}(B_1,B_2)=2$ and $\dist_{G_S}(W_1,W_2)=4$. This case is equivalent to the previous one replacing black vertices $B_i$ and $B'_i$ by white ones $W_i$ and $W'_i$, respectively.
   \item[(a3)] If $\dist_{G_S}(B_1,B_2)=4$ and $\dist_{G_S}(W_1,W_2)=4$. \\
   Then, $\dist_{G_S}(B'_1,B'_2)=4$ and $\dist_{G_S}(W'_1,W'_2)=4$. Following the ideas given above, it is easy to see that every vertex in $G$ has maximum local girth, that is, ${\bf g}(G): 6^{14}$. Therefore $G$ is a Moore graph. This graph is precisely the Heawood graph, denoted by $H$ in figure \ref{fig:examp1}.
   \item[(a4)] If $\dist_{G_S}(B_1,B_2)=2$ and $\dist_{G_S}(W_1,W_2)=2$. \\ 
    Then, $\dist_{G_S}(B'_1,B'_2)=2$ and $\dist_{G_S}(W'_1,W'_2)=2$. Notice that the sequence of vertices $(W,W_1,v,W_2)$, where $v$ is one of the black vertices $v \in \{B_1,B_2,B'_1,B'_2\}$, is a cycle of length $4$. In fact, every vertex of $S$, together with vertices $\{W,W',B,B'\}$ have local girth $4$. Then ${\bf g}(G):6^2,4^{12}$ and $G$ is isomorphic to graph $G_3$ depicted in figure \ref{fig:examp2}. 
  \end{itemize}
  \begin{itemize}
  \item[(b)] $G_S=C_4 \cup C_4$. In this case every vertex of $S$ has local girth $4$. The local girth of the vertices in the set $\{W,W',B,B'\}$ depends on the configuration of these cycles of length four. Although there are serveral cases to take into account, all of them fall into one of these three cases, after a vertex relabeling.
   \end{itemize}
   \begin{itemize}
      \item[(b1)] $\dist_{G_S}(B_1,B'_1)=2$ and $\dist_{G_S}(W_1,W'_1)=2$. \\
      Then $\dist_{G_S}(B_2,B'_2)=2$ and $\dist_{G_S}(W_2,W'_2)=2$ and there are two possibilities for $G_S$: 
      Either $G_S$ is the union of the cycles $(B_1,W_1,B'_1,W_1)$ and $(C_4:B_2,W_2,B'_2,W'_2)$ or the union of the cycles $(B_1,W_2,B'_1,W'_2)$ and $(B_2,W_1,B'_2,W'_1)$. Both are equivalent with a properly relabeling. However, each vertex of $\{W,W',B,B'\}$ has local girth $6$ in this situation. Then ${\bf g}(G):6^6,4^{8}$ and $G$ is again the graph $G_1$ depicted in figure \ref{fig:prop1}, where each $C_4$ has been depicted in different color.
   \item[(b2)] If $\dist_{G_S}(B_1,B_2)=2$ and $\dist_{G_S}(W_1,W_2)=2$. Then, $W$ (resp. $B$) has local girth $4$ since it is adjacent both to $B_1$ and $B_2$ (resp. $W_1$ and $W_2$). The same reasoning applies to $W'$ and $B'$ since $\dist_{G_S}(B'_1,B'_2)=2$ and $\dist_{G_S}(W'_1,W'_2)=2$. As a consequence ${\bf g}(G):6^2,4^{12}$ and $G$ is isomporphic to $G_4$ (See figure \ref{fig:examp2}).
   \item[(b3)] $\dist_{G_S}(B_1,B_2)=2$ and $\dist_{G_S}(W_1,W'_1)=2$. Then $W$ and $W'$ has local girth $4$. Besides $B$ and $B'$ has maximum local girth. Hence ${\bf g}(G):6^4,4^{10}$ and $G$ is isomorphic to $G_2$ depicted in figure \ref{fig:examp2}.

  \end{itemize}

\end{proof}

\begin{figure}[htb]
\centering
\begin{tabular}{ccc}
\includegraphics[width=.4\textwidth]{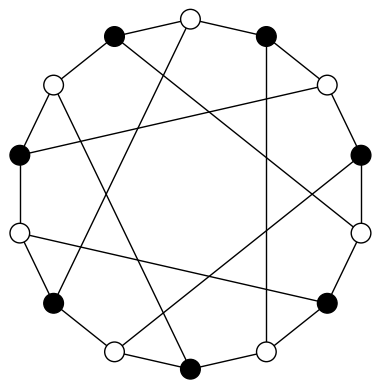} & & \includegraphics[width=.4\textwidth]{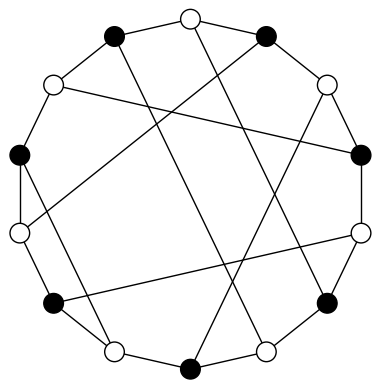} \\
Heawood Graph $H$ & & $G_1$ \\
${\bf g}(H): 6^{14}$; $\tilde{N}_1(H)=0$.   & & ${\bf g}(G_1): 6^6,4^8$; $\tilde{N}_1(G_1)=16$.
\end{tabular}
\caption{The Moore graph for $d=3$ and $g=6$ (Heawood graph) and its closest graph $G_1$ in terms of local girths.}\label{fig:examp1}
\end{figure}

Now we would like to rank each local bipartite Moore graph according to their closeness to the Moore graph. Two ranking measures are introduced in \cite{Cap10} in the context of radial Moore graphs. Here we describe the one involving local girths: Let us consider the set of graphs ${\cal LBM}(\Delta,g)$ and let $\mathbf{g}_{\Delta,g}$ be the vector of length $M^b_{\Delta,g}$ with all components equal to $g$, where $M^b_{\Delta,g}$ is the bipartite Moore bound \eqref{eq:bipgirthmoore_bound}. Notice that $\mathbf{g}_{\Delta,g}$ represent the girth vector of a bipartite Moore graph of degree $\Delta$ and girth $g$. For every positive integer $p$, 
\[
\tilde{N}_p(G)=\|\mathbf{g}(G)-\mathbf{g}_{\Delta,g}\|_p.
\]
In particular, 
\[
\tilde{N}_1(G)=\sum_{v\in V}(g-g(v))
\]
is the {\em girth norm } of $G$. Given two graphs $G_1,G_2\in {\cal LBM}(\Delta,g)$ we define $G_1$ and $G_2$ to be {\em girth-equivalent\/}, $G_1 \approx G_2$, if they have the same girth vector. In the quotient set of ${\cal LBM}(\Delta,g)$ by $\approx$, ${\cal LBM}(\Delta,g)/\approx$, we will say that $G_1$ is closer than $G_2$ to be a Moore graph if there exists a positive integer $l$ such that
\[
\tilde{N}_p(G_1)=\tilde{N}_p(G_2),\ p=1,\dots,l-1\quad \mathrm{and}\quad  
\tilde{N}_{l}(G_1)<\tilde{N}_{l}(G_2),
\]
in which case we will denote $G_1 \prec G_2$. \\
\begin{figure}[htb]
\centering
\begin{tabular}{ccc}
\includegraphics[width=.3\textwidth]{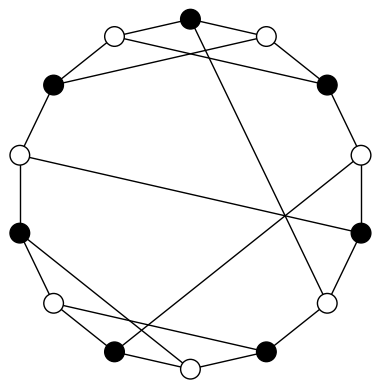} & \includegraphics[width=.3\textwidth]{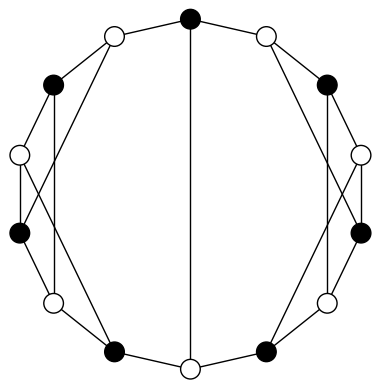}& \includegraphics[width=.3\textwidth]{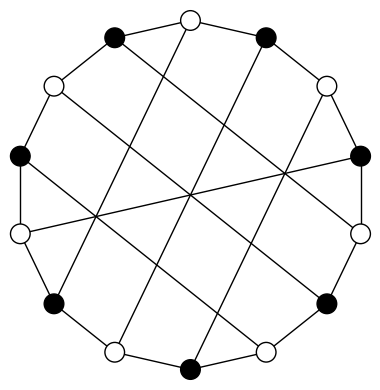}  \\
$G_2$ & $G_3$ & $G_4$  \\
${\bf g}(G_2):6^4,\ 4^{10}$ & ${\bf g}(G_3):6^2,\ 4^{12}$ & ${\bf g}(G_4):6^2,\ 4^{12}$  \\
$\tilde{N}_1(G_2)=20$ & $\tilde{N}_1(G_3)=24$ & $\tilde{N}_1(G_4)=24$ 
\end{tabular}
\caption{The remaining graphs in ${\cal LBM}(3,6)$ and their corresponding girth vector ${\bf g}(G)$ and girth norm $\tilde{N}_1(G)$. Notice that $G_3$ and $G_4$ have the same girth vector and hence both are girth-equivalent.}\label{fig:examp2}
\end{figure}

In the particular case $\Delta=3$ and $g=6$ it is suffice to compute the girth norm $\tilde{N}_1(G)$ of the five graphs to rank them. The girth norm for each one of these five graphs is calculated in proposition \ref{prop:enum} and it is resumed in figures \ref{fig:examp1} and \ref{fig:examp2}. Taking into account that $G_3$ and $G_4$ share the same vector of local girths, we have that:
\begin{corollary}
Let $H,G_1,G_2,G_3$ and $G_4$ be the five graphs in ${\cal LBM}(3,6)$. Then,
\[                                                                                                                                                                                                                                                                                                                                                                                                                                                                                                                                                                                                                                                                                                                                                                                                                                                                            
H \prec G_1 \prec G_2 \prec G_3 \approx G_4.                                                                                                                                                                                                                                                                                                                                                                                                                                                                                                                                                                                                                                                                                                                                                                                                                                                                                                                                  \]
\qed
\end{corollary}

\section{New record graphs in the context of the degree/diameter problem for bipartite graphs}\label{sec:gabi}

In this section we construct new record bipartite graphs with small defect that attend the Moore problem in the context of the degree/diameter problem. To consult more about this topic see \cite{table}. In particular, we construct a $(q+2)$-regular bipartite graph of diameter $3$ and girth $4$ adding $8$ vertices to the incidence graph of the projective plane of order $q$. See (\cite{BDF81,DF84,FPMPV13}) to consult the graphs given in \cite{table} which are attained and improved by our construction.

\subsection*{Description of the incidence graph of the projective plane or order $q$}
Firstly, we introduce the algebraic projective plane of order $q$ and their incidence graph $G_q$.
\\

Let $q$ be a prime power and denote by $GF(q)$ the Galois Field of order $q$. The following is a useful description of the projective plane over $GF(q)$: Let $P$ and $L$ two incident point and line that we will call the infinity point and line respectively, and let $\{P_0,P_1,\ldots,P_{q-1}\}$ the set of points, different of $P$, incident with $L$ and analogously $\{L_0,L_1,\ldots,L_{q-1}\}$ the set of lines, different of $L$, incident with $P$. 
\\

Moreover, let $\{(i,0),(i,1),\ldots,(i,q-1)\}$ the set of points, different to $P$, incident in $L_i$; and, let $\{[i,0],[i,1],\ldots,[i,q-1]\}$ the set of lines, different to $L$, incident with $P_i$. With the previous information, we construct the bipartite Moore tree of depth $3$, depicted on figure \ref{fig:moore_graph}, that is the spanning tree of $G_q$, the incidence graph of the projective plane of order $q$. 

The rest of the incidences are given by the cordinatization, here the line $[m,b]$ is adjacent in the graph with all the points $(x,y)$ that satisfy that $y=mx+b$ using the arithmetic of $GF(q)$. As we known, $G_q$ is a $(q+1,6)$-Moore Cage with diameter $3$ and girth $6$.\\

Figure \ref{fig:G3} depicts $G_3$, the incidence graph of the projective plane of order $3$. 

\begin{figure}[htb]
\centering
\begin{tabular}{ccc}
\includegraphics[width=.49\textwidth]{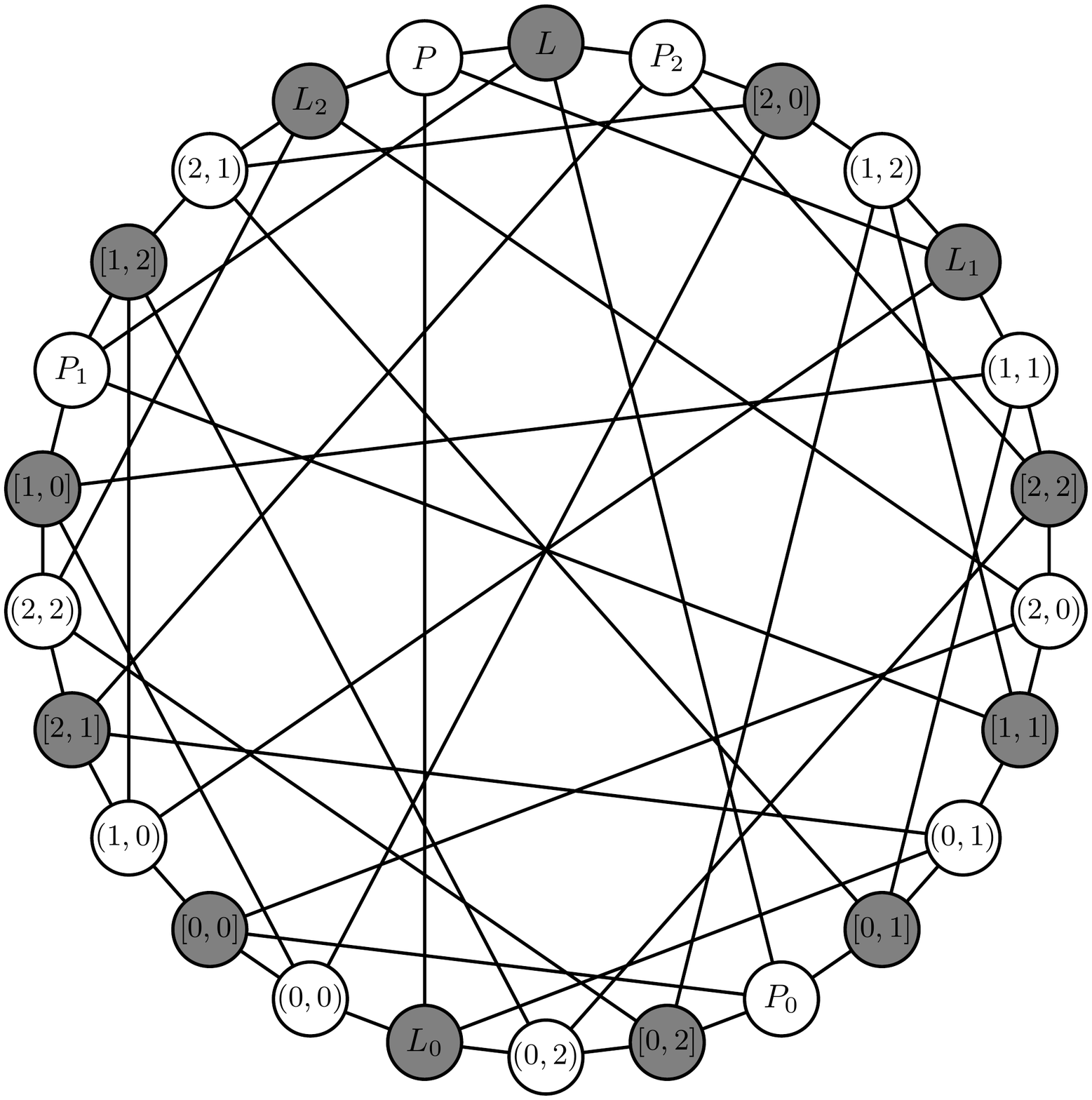} & & \includegraphics[width=.49\textwidth]{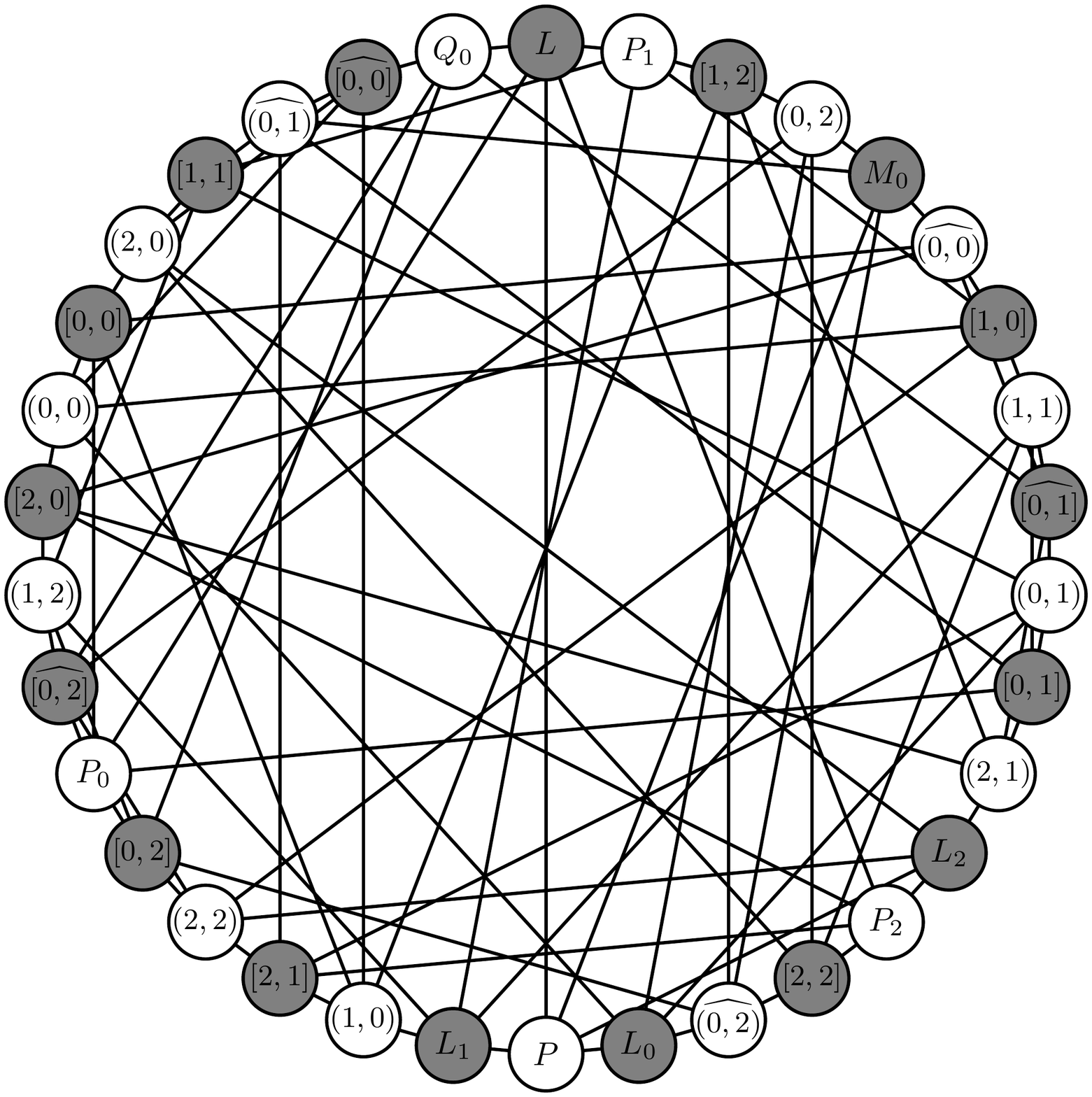} \\
$G_3$ & & $R_3$ \\
\end{tabular}
\caption{On the left $G_3$, the incidence graph of the projective plane of order $q=3$, and on the right the new graph $R_3$.}\label{fig:G3}
\end{figure}
In the sequel we construct, using $G_q$, another bipartite $(q+1)$-regular graph, called $H_q$, it has also diameter $3$, but the girth is equal to $4$. 

The next simple observation on $G_q$ is very useful for the construction of $H_q$: 

\begin{observation}\label{observation1}

\begin{itemize}
\item [(1)] The set of points incident with the line $[0,i]$, different of $L_0$, are: \newline $\{(0,i),(1,i),\ldots,(q-1,i)\}$.
\item [(2)] The set of lines incident with the point $(0,i)$, different of $P_0$, are: \newline $\{[0,i],[1,i],\ldots,[q-1,i]\}$.
\end{itemize}
\end{observation}

\subsection*{General construction for bipartite graphs of  diameter $3$}
In this section we construct $H_q$ that is $(q+2)$-regular bipartite Moore graph of diameter $3$ and girth $4$ adding $8$ vertices to $G_q$. 

Let: $$V(H_q)= V(G_q) \cup \{Q_0,M_0,\widehat{(0,0)},\widehat{(0,1)},\widehat{(0,a)},\widehat{[0,0]},\widehat{[0,1]},\widehat{[0,a]}\},$$
be the set of vertices of $H_q$, where $a$ is any element of $GF(q)$ different of $0$ and $1$ and let:

$$E(H_q)=E(H_q)- \{((0,a),[0,a])\}\cup \{(P,M_0),(L,Q_0),(Q_0,\widehat{[0,j]}),(M_0,\widehat{(0,j)}),(Q_0,[0,i]),$$
$$(M_0,(0,i)),(P_0,\widehat{[0,a]}),(L_0,\widehat{(0,a)}),(\widehat{(0,0)},\widehat{[0,1]}),(\widehat{(0,1)},\widehat{[0,0]}),(\widehat{(0,j)},[t,j]),(\widehat{[0,j]},(t,j))\}$$
be the set of edges of $H_q$, such that $j \in \{0,1,a\}$, $i \in GF(q)-\{0,1\}$ and  $t \in GF(q)$.
\\

In the sequel we prove that $H_q$ is a bipartite graph with degree sequence $\{q,q+1\}$, girth $4$ and diameter $3$:
\\

To prove that $H_q$ is bipartite we only have to note that $Q_0$, and the new vertices denoted with parenthesis are in the partite set of the points whereas $M_0$ and the new vertices denoted with brackets are in the partite set of the lines in the incidence graph of the projective plane of order $q$. To show that it has the degree sequence given, notice that the new vertices, $P$, $L$, $P_0$, $L_0$, $\{(t,j),[t,j]\}$ for $j\in \{0,1,a\}$ and $t\in GF(q)$ have degree $q+1$, and the rest has degree $q$. It is important to notice that adding the edges $(Q_0,[0,a])$ and $(\widehat{(0,a)},[0,a])$ the degree of $[0,a]$ is equal to $q+2$, but, to construct $H_q$, we delete the edge $\alpha=((0,a),[0,a])$ and with this  we garantice that $\dg_{H_q}([0,a])=q+1$. A similar analysis garantice that $\dg_{H_q}((0,a))=q+1$. 

Clearly, the girth of $H_q$ is four because it is an incidence graph, it is bipartite, and it is plenty of $4$-cycles; for instance $(M_0,(0,a),L_0,\widehat{(0,a)})$ is a $C_4$.

Related with observation \ref{observation1}, we have this new observation: 
\begin{observation}\label{observation2}
For $H_q$ we have that 
\begin{itemize}
\item [(1)] For $i\in \{0,1\}$:
\begin{itemize}
\item [] $N([0,i])\cup N(\widehat{[0,i]})=\{(0,i),(1,i),\ldots,(q-1,i)\}$,
\item [] $N((0,i))\cup N(\widehat{(0,j)})=\{[0,i],[1,i],\ldots,[q-1,i]\}$.
\end{itemize}
\item [(2)] 
\begin{itemize}
\item []$N([0,a])\cup N(\widehat{[0,a]})=\{(1,a),\ldots,(q-1,a)\}$,
\item [] $N((0,a))\cup N(\widehat{(0,a)})=\{[1,a],\ldots,[q-1,a]\}$.
\end{itemize}
\end{itemize}
\end{observation}

We are ready to prove that the diameter of $H_q$ is $3$: 

Notice that the distance between two vertices on $G_q$ is preserving in $H_q$, except the distances that involve the deleted edge $\alpha=((0,a),[0,a])$. In fact, in $H_q$, $d((0,a),[0,a])=3$ because there exists a $3$-path $((0,a),\widehat{[0,a]},L_0,[0,a])$. 

Moreover, as $\alpha^*=((0,a),\widehat{[0,a]})$ and $\alpha^{**}=([0,a],\widehat{(0,a)})$ are edges of $H_q$, by observation \ref{observation2}, the paths that use to used in $G_q$ the deleted edge $\alpha$ use in $H_q$ the edges $\alpha^*$ and $\alpha^{**}$, with this the vertices $(0,a)$ and $[0,a]$ preserve their distances with all the vertices of $G_q$.
\\

Now, we will prove that the new vertices have eccentricity at most three. 

\begin{itemize}
\item By observation \ref{observation2}, we have that $\widehat{(0,s)}$ have the same distance to all the vertices of $G_q$ than the vertex $(0,s)$, for $s\in \{0,1,a\}$, except with the vertices $\{(0,t),L_0,P\}$ for $t\in GF(q)$. 
\begin{itemize}
\item The $d(\widehat{(0,s)},(0,t))=2$ for $s\in \{0,1\}$, and $t\in GF(q)-\{0,1\}$, because there exists the path $\{\widehat{(0,s)},M_0,(0,t)\}$.
\item For $s=0$ and $t\in \{0,1\}$; $d(\widehat{(0,0)},(0,0))=2$ by the path $\{\widehat{(0,0)},[0,0],(0,0)\}$; $d(\widehat{(0,0)},(0,1))=2$ by the path $\{\widehat{(0,0)},\widehat{[0,1]},(0,1)\}$. The $d(\widehat{(0,0)},L_0)=3$ by the path $(\widehat{(0,0)},M_0,\widehat{(0,a)},L_0)$. An analogous analysis can be doing for $s=1$. 
\item For $s=a$, $L_0$ is adjacent with $\widehat{(0,a)}$ and $d(\widehat{(0,a)},(0,s))=2$ for $s\in \{0,1,a\}$ because $(\widehat{(0,a)},L_0,(0,s))$ is a path of length $2$. 

Notice that, by construction, the distance of $\widehat{(0,s)}$ to $P$ is equal to three. 
\end{itemize}

\item A similar analysis should be for $\widehat{[0,s]}$ for $s\in \{0,1,a\}$  to all the vertices of $G_q$ and for the vertices $\{[0,t],P_0,L\}$ for $t\in GF(q)$. 

\item To finish we will prove that the distance of $M_0$ to the vertices of $G_q$ is at most three, notice that their first neighborhood coincides with the first neighborhood of $L_0$ (except with $(0,s)$ for $s\in \{0,1\}$). Notice that as $\widehat{(0,s)}\in N(M_0)$ for $s\in \{0,1\}$, using  the observation \ref{observation2}, we can conclude that the second neighborhoods of $M_0$ and $L_0$
coincide in all. Then we only have to check the distance to $(0,s)$ for $s\in \{0,1\}$ and the distance to $\widehat{[0,t]}$ for $t\in \{0,1,a\}$ and $Q_0$. 
\begin{itemize}
\item The $d(M_0,(0,s))=3$ for $s\in\{0,1\}$ by the path $(M_0,(0,a),L_0,(0,s))$. 
\item The $d(M_0,\widehat{[0,0]})=d(M_0,\widehat{[0,1]})=d(M_0,\widehat{[0,a]})$ by the paths $(M_0,\widehat{(0,1)},\widehat{[0,0]})$, $(M_0,\widehat{(0,0)},\widehat{[0,1]})$, and $(M_0,(0,a),\widehat{[0,a]})$ respectively.
\item The $d(M_0,Q_0)=3$ by the path $(M_0,P,L,Q_0)$.
\end{itemize}

\item An analogously analysis could be for $Q_0$. 
\end{itemize}


Finally, to obtain a $(q+2)$-regular graph, called $R_q$, we add to $H_q$ a matching: 
\\
$\{(P_i,L_i),((x,y),[x,y])\}$ for $i\in GF(q)-\{0,1,a\}$ and $\{x,y\}\in GF(q)$ with $\dg_{H_q}\{P_i,(x,y)\}=q+1.$
This matching is well defined because, by construction if $\dg_{H_q}\{P_i,(x,y)\}=q+1.$ then also $\dg_{H_q}\{L_i,[x,y]\}=q+1.$\\

With the previous construction we have the following theorem:

\begin{theorem}\label{53bipartitegraph}
Let $q$ be a power of prime, then there exists a $(q+2)$-bipartite graph of diameter $3$ and order $2(q^2+q+5)$. 
\end{theorem}

The right side of figure \ref{fig:G3} depicts $R_3$, the $(5,3)$-biregular bipartite graph constructed before and given in theorem \ref{53bipartitegraph}.

\section{Questions and open problems}\label{sec:openproblems}

We give a complete enumeration of ${\cal LBM}(3,6)$ in section \ref{sec:measure}. It would be nice to have this enumeration also for other values of $\Delta$ and/or $g$. Nevertheless, it seems that the number of graphs in ${\cal LBM}(\Delta,g)$ increases very quickly with $\Delta$ and/or $g$, since the order of the graphs (Moore bound) follows an exponential law. 
\begin{problem}
Give a complete enumeration of ${\cal LBM}(\Delta,g)$ for any $\Delta \geq 3$ and/or $g \geq 6$, other than $(\Delta,g)=(3,6)$.
\end{problem}

There are $13$ graphs in ${\cal RB}(3,6)$ (they can be counted using an appropiatte software, like Nauty \cite{nauty}) but just $5$ of them belong to ${\cal LBM}(3,6)$. There are $23466857$ graphs in ${\cal RB}(3,8)$ (again one can use Nauty), but we do not know how many of them are local bipartite Moore graphs.
\begin{question}
Can we say something about the ratio $\dfrac{|{\cal LBM}(\Delta,g)|}{|{\cal RB}(\Delta,g)|}$?
 
\end{question}

Local bipartite Moore graphs include Moore graphs, but Moore graphs do not exist for infinitely many values of $\Delta$ and $g$. For any of these combinations of $\Delta$ and $g$ it would be nice to have the closest graph to the `theoretical' Moore graph in terms of the girth norm. For $\Delta=3$ and $g=6$ we have seen that $G_1$ is the closest graph to the Moore graph. Even for these cases when Moore graph exist, which is the `closest' graph to the Moore one? 

\begin{problem}
Find the closest graph (in terms of local girths) to the Moore graph for other values of $\Delta$ and/or $g$.
\end{problem}

In the context of degree-diameter problem of bipartite Moore graphs the principal open problems are construct graphs that improve the graphs given in \cite{table}. In particular we include two particular problems related with our results:
\begin{problem}
Find smaller bipartite regular graphs with the same parameters that we give in this paper. In other words "improve" our construction.
\end{problem}

\begin{problem}
Generalize our construction to construct bipartite regular graphs of diameters $4$ and $6$.
\end{problem} 

\subsection*{Acknowledgments}
Research of N. L\'opez was supported in part by grant MTM2017-86767-R (Spanish Ministerio de Ciencia e Innovacion) and research of G. Araujo was supported by PASPA-DGAPA Sabatical Year 2020, CONACyT-M{\' e}xico under Project 282280 and PAPIIT-M{\' e}xico under Projects IN107218, IN106318. G. Araujo would like to thank Ruben Alfaro 
for his help and computation assistance at the first of the construction on the graph given in Section \ref{sec:gabi}.


\end{document}